\documentclass[a4paper,english,cleveref, autoref]{lipics-v2021}
\usepackage{amsmath, amsfonts,amsthm, amssymb}
\usepackage{mathrsfs}
\usepackage[utf8]{inputenc}
\newbox\myboxa
\newbox\myboxb
\def\clap#1{\hbox to 0pt{\hss#1\hss}}

\def\mathclap{\mathpalette\mathclapinternal}

\def\mathclapinternal#1#2{%
           \clap{$\mathsurround=0pt#1{#2}$}}

\newtheorem{problem}[theorem]{Problem}

\newcommand{\N}{\mathbb{N}}

\newcommand{\adm}{\mathrm{adm}}
\newcommand{\col}{\mathrm{scol}}
\newcommand{\wcol}{\mathrm{wcol}}

\DeclareMathOperator{\tww}{\rm{tww}}

\DeclareMathOperator{\bomega}{\rm{b}\omega}
\DeclareMathOperator{\Wreach}{{\rm WReach}}
\DeclareMathOperator{\Sreach}{{\rm Sreach}}
\DeclareMathOperator{\birth}{{\rm bt}}
\DeclareMathOperator{\spl}{{\rm st}}
\newenvironment{clproof}{ \trivlist
  \item[\hskip\labelsep
        \emph{Proof of the claim}.]\ignorespaces
}{\hfill$\vartriangleleft$\medskip

}

\title{Twin-width and generalized coloring numbers}

\author{Jan Dreier}{Vienna University of Technology}{dreier@ac.tuwien.ac.at}{}{}%
\author{Jakub Gajarsky}{University of Warsaw}{gajarsky@mimuw.edu.pl}{}{}
\author{Yiting Jiang}{Universit\'e de Paris, CNRS, IRIF, F-75006, Paris, France and Department of Mathematics, Zhejiang Normal University, China}{yjiang@irif.fr}{}{}%
\author{Patrice Ossona de Mendez}{Centre d'Analyse et de Math\'ematiques Sociales (CNRS, UMR 8557), Paris, France \and Computer Science Institute of Charles University, Praha, Czech Republic}{pom@ehess.fr}{https://orcid.org/0000-0003-0724-3729}{}%
\author{Jean-Florent Raymond}{CNRS, LIMOS, Université Clermont Auvergne, France}{j-florent.raymond@uca.fr}{https://orcid.org/0000-0003-4646-7602}{}
\authorrunning{J.\ Dreier, J.~Gajarsky, Y.\ Jiang, P.\ Ossona de Mendez, and J.-F.\ Raymond}
\Copyright{Jan Dreier, Jakub Gajarsky, Yiting Jiang, Patrice Ossona de Mendez, and Jean-Florent Raymond}



\category{} 

\relatedversion{} 

\supplement{}

\newcommand{\ERCagreement}{{\begin{minipage}{.56\textwidth}This paper is part of a project that has received funding from the European Research Council (ERC) under the European Union's Horizon 2020 research and innovation programme (grant agreement No 810115 -- {\sc Dynasnet}). \end{minipage}\hfill\begin{minipage}{.33\textwidth}\includegraphics[width=\textwidth]{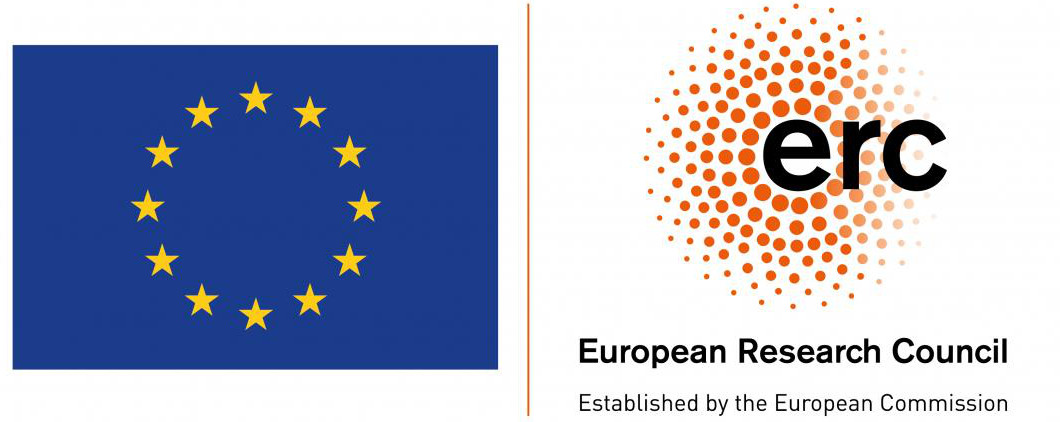}\end{minipage}\hfill}}

\funding{\ERCagreement}
\acknowledgements{}

\nolinenumbers 

\hideLIPIcs  

\EventEditors{John Q. Open and Joan R. Access}
\EventNoEds{2}
\EventLongTitle{42nd Conference on Very Important Topics (CVIT 2016)}
\EventShortTitle{CVIT 2016}
\EventAcronym{CVIT}
\EventYear{2016}
\EventDate{December 24--27, 2016}
\EventLocation{Little Whinging, United Kingdom}
\EventLogo{}
\SeriesVolume{42}
\ArticleNo{23}

\begin{document}

\maketitle

\begin{abstract}
    In this paper, we prove that a graph $G$ with no $K_{s,s}$-subgraph and twin-width $d$ has $r$-admissibility and $r$-coloring numbers bounded from above by an exponential function of $r$ and that we can construct graphs achieving such a dependency in $r$.
\end{abstract}

\section{Introduction}
In this paper we consider the twin-width graph parameter, defined by
Bonnet, Kim, Thomass\'e and Watrigant~\cite{twin-width1}
as a generalization of a width invariant for classes of permutations
defined by Guillemot and Marx~\cite{Guillemot14}. This parameter was
intensively studied recently in the context of many structural and
algorithmic questions such as FPT model checking~\cite{twin-width1},
graph enumeration~\cite{twin-width2}, graph
coloring~\cite{twin-width3}, matrices and ordered
graphs~\cite{twin-width4}, and transductions of permutations~\cite{TWWP-arxiv}. (We postpone the formal definition of twin-width to \Cref{sec:tww}.)

It is known that a graph class with bounded twin-width excludes some biclique as a subgraph if and only if it has bounded expansion~\cite{twin-width2}.
 Recall that a class $\mathscr C$ has \emph{bounded expansion} if, for each integer $r$ the class of all the minors of graphs of $\mathscr C$ obtained by contracting vertex disjoint connected subgraphs with radius at most $r$ and deleting some edges and vertices have bounded average degree, which may depend on $r$.  
(We refer the interested reader to~\cite{Sparsity} for an in-depth study of classes with bounded expansion.) Among the numerous characterizations of classes with bounded expansion, three relate to the 
\emph{generalized colouring numbers} $\wcol_r$ and $\col_r$ introduced by Kierstead and Yang~\cite{Kierstead2003} and to the \emph{$r$-admissibility} $\adm_r$ introduced by Dvo\v r\'ak~\cite{Dvovrak2011}. Indeed, as proved by Zhu~\cite{Zhu2008}, the following are equivalent for a class $\mathscr C$:
\begin{enumerate}
    \item $\mathscr C$ has bounded expansion;
    \item $\sup\{\wcol_r(G): G\in\mathscr C\}<\infty$ for every integer $r$;
    \item $\sup\{\col_r(G): G\in\mathscr C\}<\infty$ for every integer $r$.
\end{enumerate}
Moreover, using the inequality $\adm_r(G)\leq \col_r(G)\leq \wcol_r(G)\leq \frac{\adm_r(G)^{r+1}-1}{\adm_r(G)-1}$ (see \cite{Dvovrak2011}), we get yet another equivalent property.
\begin{enumerate}
\setcounter{enumi}{3}
    \item $\sup\{\adm_r(G): G\in\mathscr C\}<\infty$ for every integer $r$.
\end{enumerate}

One can show~\cite{twin-width2} that for every integer $r$ there exists a function $f_r:\mathbb N\times \mathbb N\rightarrow \mathbb N$ such that if $G$ is a graph with twin-width $t$ and no $K_{s,s}$-subgraph, then we have $\wcol_r(G)\leq f_r(t, s)$. Similar bounds also exist for $\col_r$ and $\adm_r$. 
However, the proof given in \cite{twin-width2} that biclique-free graphs with bounded twin-width have bounded expansion does not indicate how to compute such binding functions. 

In this paper, we prove that a graph $G$ with no $K_{s,s}$-subgraph and twin-width $d$ has $\adm_r$, $\col_r$ and $\wcol_r$ bounded from above by an exponential function of $r$, and that we can construct graphs achieving such a dependency in $r$.
In particular, $\col_r(G) \le (d^r+3)s$~(\Cref{thm:upperbound}).
On the other hand, one can choose $G$ such that $\col_r(G) \ge (\frac{d-4}{8})^rs$~(\Cref{cor:lowerboundscol}).

\section{Definitions and Notations}
\subsection{Twin-width}
\label{sec:tww}

We define twin-width with the help of \emph{trigraphs}.
The notion of trigraphs used in this work is slightly different from the notion used in \cite{twin-width1}. 
Both notions are nevertheless equivalent up to isomorphism.
A {\em trigraph} $\mathbf G$ \emph{on a graph} $G=(V,E)$ is a binary structure with two binary relations, the \emph{black adjacency} $E$ and the \emph{red adjacency} $R$, whose domain is a partition of $V$, and whose black and red adjacencies are exclusive (that is: no two elements of $\mathbf G$ can be adjacent in both relations). 
Thus, the elements of $\mathbf G$ are subsets of the vertices of $G$.
To distinguish the elements of $\mathbf G$ from the vertices of $G$, we will call them \emph{nodes} and denote them by capital letters, like $X,Y,Z$. The set of nodes of $\mathbf{G}$ is denoted by $V(\mathbf{G})$.
The elements of  $E(\mathbf G)$ and $R(\mathbf G)$ are respectively called \emph{black edges} and \emph{red edges}. 
The {\em set of neighbours} $N_{\mathbf G}(X)$ of a node $X$ in a trigraph $\mathbf G$ consists of all the nodes adjacent to $X$ by a black or red edge; the {\em set of $E$-neighbours} $N_{\mathbf G}^E(X)$ consists of all nodes adjacent to $X$ by a black edge and the {\em set of $R$-neighbours} $N_{\mathbf G}^R(X)$ consists of all nodes adjacent to $X$ by a red edge. A {\em $d$-trigraph} is a trigraph $\mathbf G$ with maximum red degree at most $d$, i.e., $|N_{\mathbf G}^R(X)|\leq d$ for all $X\in V(\mathbf G)$.
 
 Let $\mathbf G$ be a trigraph on a graph $G$ and let $X$ and $Y$ be (non-necessarily adjacent) nodes of $\mathbf G$. 
We say a trigraph $\mathbf G'$ on $G$ is obtained from $\mathbf G$ by {\em contracting} $X$ and $Y$ if $V(\mathbf G')=V(\mathbf G)\setminus\{X,Y\}\cup\{X\cup Y\}$, 
$N_{\mathbf G'}(X\cup Y)=N_{\mathbf G}(X)\cup N_{\mathbf G}(Y)$, $N_{\mathbf G'}^E(X\cup Y)=N_{\mathbf G}^E(X)\cap N_{\mathbf G}^E(Y)$ (and  $N^R_{\mathbf G'}(X\cup Y)=N_{\mathbf G'}(X\cup Y)\setminus N^E_{\mathbf G'}(X\cup Y)$), and the red and black adjacencies between all other nodes of $\mathbf G'$ are as in $\mathbf G$.

A {\em $d$-contraction sequence} of a graph $G=(V,E)$ with $n$ vertices is a sequence $\mathbf G_{n},\dots,\mathbf G_1$ of $d$-trigraphs on $G$, where $\mathbf G_{n}$ is the trigraph isomorphic to $G$ defined by
	$V(\mathbf G_n)=\{\{v\}:v\in V\}$, $E(\mathbf G_n)=\{(\{u\},\{v\}): (u,v)\in E(G)\}$, and $R(\mathbf G_n)=\emptyset$, $\mathbf G_1$ is the trigraph with single node $V$, and $\mathbf G_{i}$ is obtained from $\mathbf G_{i+1}$ by performing a single contraction. 
The minimum $d$ such that there exists a $d$-contraction sequence of a graph $G$ is the {\em twin-width} of $G$, denoted by $\tww(G)$. 	
For a contraction sequence $\mathbf G_{n},\dots,\mathbf G_1$,
we define the \emph{universe} $\mathcal U=\bigcup_{i=1}^n V(\mathbf G_i)$ to be the union of all node sets.

A given contraction sequence $\mathbf G_n,\dots,\mathbf G_1$ on a graph $G=(V,E)$ (with universe $\mathcal U$) can also be reversed to $\mathbf{G}_1, \dots, \mathbf{G}_n$ and seen as an \emph{uncontraction sequence} where we start with a single node (the trigraph $\mathbf{G}_1$) and a node
$Z$ of $\mathbf G_i$ is split into two nodes $X$ and $Y$ with no edge, black edge or red edge between them in $\mathbf G_{i+1}$.
With this picture in mind, we define for every $X\in\mathcal U$, the \emph{birth time} $\birth(X)$ as the minimum integer $i$ with $X\in V(\mathbf G_i)$ and the  \emph{split time} $\spl(X)$ as the maximum integer $i$ with $X\in V(\mathbf G_i)$. 
Observe that for every $i \in \{1, \dots, n-1\}$, there is a unique $X \in \mathcal{U}$ with $\spl(X) = i$; the subsets $X\in \mathcal{U}$ with $\spl(X)=n$ are  the nodes of $\mathbf G_n$, that is the singletons $\{v\}$ with $v\in V(G)$.
If $X\in\mathcal U\setminus\{V\}$, the \emph{parent} of $X$ is the minimal set $Y\in\mathcal U$ with $Y\supsetneq X$. Conversely, if $|X|>1$, the \emph{children} of $X$ are the two maximal sets $Y$ and $Z$ in $\mathcal U$ with $Y\subsetneq X$ and $Z\subsetneq X$. Note that $\{Y,Z\}$ is a partition of $X$ and that $\birth(Y)=\birth(Z)=\spl(X)+1$. 

\subsection{Generalized Colouring Numbers and Admissibility}
\label{sec:col}
Let $\Pi(G)$ be the set of all linear orders of the vertices of the graph
$G$, and let $L\in\Pi(G)$.
(We denote by $\leq_L$ the corresponding binary relation for better readability.)
Let $u,v\in V(G)$, and let $r$ be a positive integer.

We say that $u$ is
\emph{weakly $r$-reachable} from~$v$ with respect to~$L$, if there exists a
path $P$ of length at most $r$ between $u$ and $v$ such that
$u \le_L w$ for all vertices $w$ of $P$. Let
$\Wreach_r[L,v]$ be the set of vertices that are weakly $r$-reachable
from~$v$ with respect to $L$. Note that $v\in\Wreach_r[L,v]$.

We say that $u$ is \emph{strongly $r$-reachable} from $v$ with respect to~$L$, if
there is a path $P$ of length at most $r$ connecting $u$ and
$v$ such that $u\leq_Lv$ and all inner vertices $w$ of $P$ satisfy
$v<_Lw$. Let $\Sreach_r[L,v]$ be the set of vertices that are strongly
$r$-reachable from~$v$ with respect to $L$. Note that again we have
$v\in \Sreach_r[L,v]$.

The \emph{$r$-backconnectivity} $b_r(L,v)$ of a vertex $v$ is the maximum number of paths of length at most $r$ in $G$ that start in $v$, 
share no other vertices except $v$, and end at vertices that lie before $v$ in the ordering $L$.

The \emph{weak $r$-colouring number
  $\wcol_r(G)$} of $G$ is defined as
\begin{equation*}
  \wcol_r(G):=\min_{L\in\Pi(G)}\:\max_{v\in V(G)}\:
  \bigl|\Wreach_r[L,v]\bigr|,
\end{equation*}
and the \emph{strong $r$-colouring number $\col_r(G)$} of $G$ is defined as
\begin{equation*}
  \col_r(G):=\min_{L\in\Pi(G)}\:\max_{v\in V(G)}\:
  \bigl|\Sreach_r[L,v]\bigr|.
\end{equation*}

The \emph{$r$-admissibility} of $G$ is defined as
\begin{equation*}
    \adm_r(G)=\min_{L\in\Pi(G)}\:\max_{v\in V(G)}\:b_r(L,v).
\end{equation*}

\section{From Strong Colouring to Weak Colouring}

It is known that the weak and strong colouring numbers are related by $\col_r(G)\leq \wcol_r(G)\leq \col_r(G)^r$~\cite{Kierstead2003}.
However it is possible to improve the upper bound in the case where the strong coloring numbers increase at least at an exponential rate.

\begin{lemma}\label{lem:st-to-wk}

For every graph $G$ and every positive integer $r$ we have
\[
\wcol_r(G)\leq 2^{r-1}\max_{1\leq k\leq r} \col_k(G)^{r/k}.
\]
\end{lemma}
\begin{proof}
Let $r$ be a positive integer and let $L$ be a linear order on $V(G)$ that minimizes
\[
    \max_{v\in V(G)}\: \bigl|\Wreach_r[L,v]\bigr|.
\]
Let $u$ be a vertex of $G$, $v\in \Wreach_r[L, u]$,  and consider a path $P$ certifying that $v$ is weakly $r$-reachable from $u$; in particular $P$ has length at most $r$.
Let $C(r)$ be the set of all \emph{compositions} of $r$, that is of all tuples $(r_1,\dots,r_k)$ with $r_i>0$ (for $1\leq i\leq k$) and $\sum_{1\leq i\leq k}r_i=r$.
A \emph{milestone} of $P$ is a vertex 
$v$ of $P$ such that all the vertices of $P$ from $u$ (included) to $v$ (excluded) are greater than $v$. 
Let $v_1,\dots,v_k=v$ be the milestones of $P$ other than $u$, and let $r_1,\dots,r_{k-1}$ be the lengths of the paths from $v_0=u$ to $v_1$,\dots, $v_{k-2}$ to $v_{k-1}$, and let $r_k=r-\sum_{i=1}^{k-1}r_i$, so that $(r_1,\dots,r_k)\in C(r)$.
The subpath of $P$ from $v_{i-1}$ to $v_i$ witnesses that $v_i$ is strongly $r_i$-reachable from $v_{i-1}$. Note that strong $r_k$-reachability requires the existence of a witness path of length at most $r_k$, hence it is safe to consider $r_k$ instead of the length of the subpath of $P$ linking $v_{k-1}$ and $v_k$. 
 We deduce that
\[
\Wreach_r[L,u]\subseteq\quad\bigcup_{\mathclap{(r_1,\dots,r_k)\in C(r)}}\quad\qquad\bigcup_{v_1\in\Sreach_{r_1}[L,u]}\quad\cdots\quad\qquad\bigcup_{\mathclap{v_{k-1}\in\Sreach_{r_{k-1}}[L,v_{k-2}]}}\quad\Sreach_{r_k}[L,v_{k-1}].
\]
(Note that we actually have equality, the reverse inclusion following from the concatenation of paths witnessing $v_1\in\Sreach_{r_1}[L,v]$, \dots, $u\in\Sreach_{r_k}[L,v_{k-1}]$.)
Thus we have
\[
\wcol_r(G)\leq \sum_{(r_1,\dots,r_k)\in C(r)}\prod_{i=1}^k \col_{r_i}(G).
\]
Let $z=\max_{1\leq k\leq r}\col_{k}(G)^{1/k}$. 
Then $\col_{r_i}(G)\leq z^{r_i}$. Thus
\[
\wcol_r(G)\leq \sum_{(r_1,\dots,r_k)\in C(r)}\prod_{i=1}^k z^{r_i}= |C(r)|\, z^r=2^{r-1}z^r.
\]
\end{proof}

\section{Upper bounds}

Let $\bomega(G)$ denote the maximum integer $s$ such that
$K_{s,s}$ is a subgraph of $G$.
\begin{theorem}\label{thm:upperbound}
	For every graph $G$ and every positive integer $r$ we have 
	\begin{equation}
	\label{eq:ub}
			\col_r(G)\le \biggl(3+\tww(G)\sum_{i=0}^{r-1}(\tww(G)-1)^{i}\biggr)\bomega(G)\le  
 (\tww(G)^r+3)\bomega(G).
 \end{equation}
\end{theorem}
\begin{proof}
	Let $d=\tww(G)$ and $s=\bomega(G)$.
	Without loss of generality, we can assume that $G$ is connected and contains more than $s$ vertices.
	We consider a $d$-uncontraction sequence $\mathbf G_1,\dots,\mathbf G_n$ of $G$ with universe $\mathcal U$. 
    For every $i \in \{1, \dots, n\}$ we say a node of $\mathbf G_i$ is \emph{small} if it contains at most $s$ vertices and it is \emph{big}, otherwise.
	A set $X\in\mathcal U$ is \emph{nice at step} $i$ with $\birth(X) \leq i \leq \spl(X)$ if $X$ is small and some black edge is incident to it in $\mathbf G_{i}$. 
    Note that if $X$ is nice at step $i$, it is nice at step $j$ for all $i\leq j\leq\spl(X)$.
	The set $X$ is \emph{nice} if it is nice at some step (equivalently, at step $\spl(X)$). 
    For every nice set $X$ we define $\rho(X)$ as the minimum $i$ such that $X$ is nice at step~$i$. (Note that $\rho(X)>1$ as $\mathbf G_1$ is edgeless.)
	As $G$ is connected, it is clear that every $X\in\mathcal U$ has a subset $Y\in\mathcal U$ that is nice. Also, if $X,Y\in\mathcal U$, $X\subseteq Y$ and $Y$ is nice, then $X$ is also nice. 
	It follows that the family $\mathcal N$ of all the maximal nice sets in $\mathcal U$ form a partition of $V$. 
	We order the elements of $\mathcal N$ as $N_1,\dots,N_k$ in such a way that for all $i<j$, either $\rho(N_i)<\rho(N_j)$ holds or $\rho(N_i)=\rho(N_j)$ and $\birth(N_i)\ge\birth(N_j)$. 
    We now fix any linear ordering $L$ of $V$ such that for all $v \in N_i$, $w \in N_j$ with $i < j$ holds $v <_L w$.
    See also \Cref{remark} for an equivalent algorithmic way to define the order $L$.
    We will  use this ordering to bound the strong coloring numbers of $G$.

	For $1 \le i \le n$, we define $\mathcal B_i$ to be the set of all nodes of $\mathbf G_i$ that are not nice at step~$i$. We first establish some easy properties of $\mathcal B_i$.
	\begin{claim}
		\label{cl:1}
		No small node in $\mathcal B_i$ is incident to a black edge in $\mathbf G_i$.
	\end{claim}
	\begin{clproof}
		Assume $X\in\mathcal B_i$ is small. Then it is not adjacent to a black edge as it is not nice at step $i$.
	\end{clproof}
	\begin{claim}
		\label{cl:2}
		No two nodes in $\mathcal B_i$ are adjacent in $\mathbf G_i$ by a black edge.
	\end{claim}
	\begin{clproof}
		Assume for contradiction that $X$ and $Y$ are nodes in $\mathcal B_i$ that are adjacent by a black edge in $\mathbf G_i$.
        Hence, $G[X\cup Y]$ includes $K_{|X|,|Y|}$ as a subgraph.
		According to \Cref{cl:1}, both $X$ and $Y$ are big, contradicting the assumption $\bomega(G)=s$.
	\end{clproof}
	\begin{claim}
        \label{cl:3}
  For every node $X \in \mathcal B_i$ holds $|\bigcup N_{\mathbf G_i}^E(X)| \le s$.
	\end{claim}
	\begin{clproof}
		Let $X\in\mathcal B_i$ and let $Y=\bigcup N_{\mathbf G_i}^E(X)$. Then $X$ and $Y$ induce a biclique in $G$ thus $\min(|X|,|Y|)\leq s$. As only big nodes in $\mathcal B_i$ are adjacent to black edges (by \Cref{cl:1}), we deduce $|X| > s$ and therefore  $|Y|\leq s$.
	\end{clproof}

    Let us consider a vertex $v \in V$.
    In the remainder of the proof, we will bound the number of vertices in $G$ that are strongly $r$-reachable from $v$ with respect to~$L$.
    Let $a \in \{1, \dots, k\}$ be such that $v \in N_a$ and let $t=\rho(N_a)-1$. 
    Let $S$ be the unique node of $\mathbf G_t$ with $\spl(S)=t$, and let $Y,Z$ be the two children of $S$.

    Let $\mathcal L=\{X\in V(\mathbf G_t): (\exists i<a), N_i\supseteq X\}$. Then $\mathcal L$ is the set of all nodes of $\mathbf G_t$ that are nice at step $t$. By definition of $L$, all the vertices of $G$ that belong to these nodes appear before $v$ in~$L$. If we set $\mathcal R=V(\mathbf G_t)\setminus\mathcal L$, then $\mathcal{R}=\mathcal B_t$.

    \paragraph*{Case 1: $v\in S$.}
	Note that if a vertex $u \in V(G)$ is strongly $r$-reachable from $v$, then $u$ belongs either to $S$ or to a set in $\mathcal L$.
    We consider a BFS-tree $T$ in $\mathbf G_t$, starting at $S$, following only red edges, with depth $r$, and stopping each time it reaches a node in $\mathcal L$. We further remove from $T$ any node with no descendant (in $T$) belonging to $\mathcal L$. This way we get a tree $T$ rooted at $S$, with internal nodes in $\mathcal R$, with depth at most $r$, and with leaves in $\mathcal L$. Let $\mathcal I$ denote the sets of all internal nodes of $T$ and let $\mathcal E$ be the sets of all leaves of $T$. 
    Then $|\mathcal I|\leq 1+\sum_{\ell=0}^{r-2}d(d-1)^\ell$ and $|\mathcal E|\leq d(d-1)^{r-1}$. 
    Consider any vertex $u$ that is strongly $r$-reachable from $v$, and let $P$ be a path from $v$ to $u$ in $G$ witnessing this.
    We can project $P$ onto $\mathbf G_t$ by mapping every vertex to the node of $\mathbf G_t$ containing it.
    The projection is a walk from $S$ to a node $X_u$ containing $u$. 
    From this walk we extract a path $Q$ of length at most $r$ from $S$ to $X_u$. 
    All the internal nodes of $Q$ as well as $S$ belong to $\mathcal R$, hence all the edges of $Q$ (but maybe the last one) are red (according to \Cref{cl:2}). Moreover, $X_u$ is either $S$ or it belongs to $\mathcal L$.
    So, either $u\in S$, or $X_u$ has been reached by a black edge from some internal node of $T$, or $X_u$ is a leaf of $T$. It is easily checked that at most $|\mathcal I|s$ vertices of $G$ can be of the second type (according to \Cref{cl:3}), and at most $|\mathcal E|s$ are of the last type (as leaves belong to $\mathcal{L}$, so they are small).
    Regarding the first type, we assume without loss of generality that $v \in Z$ and observe that either $u \in Z$, so there are at most $s$ choices for $u$ (as $Z = N_a$ is nice hence small at time $t+1$), or $u\in Y$ but then, as $u \leq_L v$, $Y$ is nice as well at time $t+1$ so $|Y| \leq s$. Thus at most $2s$ vertices of $G$ can be of the first type.
    Altogether, we get
\begin{align*}
	|\Sreach[G,L,v]|&\leq \left (2+1+d+\dots+d(d-1)^{r-2}+d(d-1)^{r-1}\right)s\\
	& \leq \left (3 + d \sum_{\ell=0}^{r-1} (d-1)^\ell \right )s.
\end{align*}
	
    \paragraph*{Case 2: $v\notin S$.} 
	Note that if $u$ is strongly $r$-reachable from $v$, then $u$ belongs either to $S$ or to $N_a$, or to a set in $\mathcal L$.
	We consider a BFS-tree $T$ in $\mathbf G_t$, starting at $N_a$, following only red edges, with depth $r$, and stopping each time it reaches a node in $\mathcal L$. We further remove from $T$ any node with no descendant (in $T$) belonging to $\mathcal L\cup\{S\}$.
	This way we get a tree $T$ rooted at $N_a$, with internal nodes in $\mathcal R$, with depth at most $r$, and with leaves in $\mathcal L$. Let $\mathcal I$ denote the sets of all internal nodes of $T$ and let $\mathcal E$ be the sets of all leaves of $T$. 
	Then $|\mathcal I|\leq 1+d+\dots+d(d-1)^{r-2}$ and $|\mathcal E|\leq d(d-1)^{r-1}$. Consider any vertex $u$ that is strongly $r$-reachable from $v$, and let $P$ be a path from $v$ to $u$ witnessing this. The path $P$ projects on $\mathbf G_t$ as a walk with length at most $r$ from $N_a$ to the vertex $X_u$ containing $u$. From this walk we extract a path $Q$ with length at most $r$ from $N_a$ to $X_u$. All the internal nodes of $Q$ belong to $\mathcal R$ hence all the edges of $Q$ (but maybe the last one) are red (according to \Cref{cl:2}). Moreover, $X_u$ is either $N_a$, or $S$, or it belongs to $\mathcal L$.
		So, either $u\in N_a$, or $u\in S$, 
		or $X_u$ has been reached by a black edge from some internal node of $T$, or $X_u$ is a leaf of $T$.
		The first type correspond to at most $s$ vertices. The second type correspond to at most $2s$ vertices
		because in this case, $Y$, $Z$, or both, have been ordered by $L$ before $N_a$ which mean they are nice at step $t+1$, hence small.
		. The third type correspond to at most $(|\mathcal I|-1)s$, as the root $N_a$ is small hence adjacent to no black edge. The last type correspond to at most $|\mathcal E|s$ vertices.
		Altogether, we get
\begin{align*}
	|\Sreach[G,L,v]|&\leq \left (1+2+(1+d+\dots+d(d-1)^{r-2}-1)+d(d-1)^{r-1}\right)s\\
	&\leq \left (3 + d\sum_{\ell=0}^{r-1}(d-1)^\ell \right)s.
\end{align*}

Thus in both cases we have that every graph $G$ with $\tww(G)=d$ and $\bomega(G)=s$ satisfies~\eqref{eq:ub}.
\end{proof}

\newcommand\ori{\textnormal{origin}}
\begin{remark}\label{remark}\rm
    We also describe an algorithmic procedure that also yields the order $L$ defined in~\Cref{thm:upperbound}.
    We are given an uncontraction sequence $\mathbf G_n,\dots,\mathbf G_1$.
    For each $i \in \{1, \dots, n-1\}$ we define a function $\ori_i \colon V(\mathbf{G}_{i+1}) \to V(\mathbf{G}_i)$ that, informally, assigns each node in $\mathbf G_{i+1}$ to the node in $\mathbf G_i$ that it originates from.
    Let us be more precise: assume $\mathbf G_{i+1}$ is constructed from $\mathbf G_i$ by splitting a node $Z$ into two nodes $X$ and $Y$, then
    $\ori_i(X) = \ori_i(Y) = Z$ and $\ori_i(W) = W$ for every other node $W$ of $\mathbf{G}_{i+1}$.

    Remember that a node of $\mathbf G_i$ is \emph{small} if it contains at most $s$ vertices and is \emph{big}, otherwise.
    A node $X\in V(\mathbf G_i)$ is a \emph{nice} node of $\mathbf G_i$ if $X$ is small and some black edge is incident to $X$ in $\mathbf G_{i}$.
    We incrementally construct for all $i$ an ordering $<_i$ on the nice nodes of $G_i$.
    Since all nodes of $\mathbf G_n$ are nice and correspond to singletons,
    the ordering $<_n$ then corresponds to an ordering of the vertices of $G$.
    This order will be equivalent to the ordering $L$ defined in \Cref{thm:upperbound} (up to non-determinism).

    Since $G_1$ has no nice nodes, $<_1$ is the empty ordering.
    Assuming that $<_i$ is already constructed,
    we construct $<_{i+1}$ such that it satisfies for all nice $X, Y \in V(\mathbf G_{i+1})$
    the following conditions.
    \begin{enumerate}
        \item If $\ori_i(X)$ and $\ori_i(Y)$ are nice in $\mathbf G_i$ and $\ori_i(X) <_i \ori_i(Y)$ then $X <_{i+1}~Y$.
        \item If $\ori_i(X)$ is nice in $\mathbf G_i$ and $\ori_i(Y)$ is not nice in $\mathbf G_i$ then $X <_{i+1}~Y$.
        \item If both $\ori_i(X)$ and $\ori_i(Y)$ are not nice in $\mathbf G_i$ and $\birth(X) > \birth(Y)$ then $X \le_{i+1} Y$.
    \end{enumerate}
    Each order $<_i$ represents a partial order on $V$ that is refined over time as $i$ increases,
    until we reach a total order on $V$.
    Rule 1.\ states that the old order is preserved when possible, 
    rule 2.\ states that new nice sets are appended at the end
    and rule 3.\ makes sure that we append new nice sets in order of their birth.

    In the proof of \Cref{thm:upperbound}, we fix a vertex $v \in V$
    and pick $t$ maximal such that in $\mathbf G_t$ the node $N$ containing $v$ is not nice.
    We then partition the nodes of $\mathbf G_t$ into sets $\mathcal L$ and $\mathcal R$.
    One can show that $\mathcal L$ contains precisely those nodes of $\mathbf G_t$ that are strictly smaller than $N$ with respect to $<_t$.
\end{remark}

\begin{corollary}
	For every graph $G$ and every positive integer $r$ we have 

\[
\col_r(G)\leq\begin{cases}
	2\bomega(G)&\text{if $\tww(G)=0$,}\\
	3\bomega(G)&\text{if $\tww(G)=1$,}\\
	5\bomega(G)&\text{if $\tww(G)=2$,}\\
	3(\tww(G)-1)^r \bomega(G)&\text{if $\tww(G)\geq 3$.}
\end{cases}
\]
\end{corollary}
\begin{proof}
If $\tww(G)\geq 3$ we have
\begin{align*}
\col_r(G) & \le \left (3+\tww(G)\frac{(\tww(G)-1)^r - 1}{\tww(G) - 2}\right)\bomega(G)\\
          & \leq (3 + 3((\tww(G)-1)^r - 1))\bomega(G)\\
          & \leq 3(\tww(G)-1)^r \bomega(G).
\end{align*}
The cases where $\tww(G)=1$ or $2$ follow from the theorem. If $\tww(G)=0$ then $G$ is a cograph. Let us then show that for every cograph $G$ it holds that $\col_r(G) \leq 2 \bomega(G)$.

The proof is by induction on the number of vertices. The base case $|G| = 1$ is trivial so we consider a cograph with at least two vertices and assume that the desired bound holds for all cographs on less vertices.
Being a cograph, $G$ can be obtained from two cographs $G_1$ and $G_2$ by disjoint union or complete join~\cite{corneil1981complement}. Without loss of generality we assume $|G_1| \leq |G_2|$. By induction, for every $i\in\{1,2\}$ there is an ordering $L_i$ of $V(G_i)$ such that $\col_r(G_i) \leq 2\bomega(G_i)$.

Then the order $L$ is obtained by putting first $L_1$, then $L_2$.
If $G$ is the disjoint union of $G_1$ and $G_2$ then $\col_r(G,L)=\max(\col_r(G,L_1),\col_r(G,L_2))$ and the result follows; if $G$ is the complete join of $G_1$ and $G_2$ then $\col_r(G,L)\leq \col_r(G_2,L_2)+|G_1|$ and $\bomega(G)\geq \bomega(G_2)+|G_1|/2$. As $\col_r(G_2,L_2)\leq 2\bomega(G_2)$, we get $\col_r(G_2,L_2) \leq 2\bomega(G)-|G_1|$ hence $\col_r(G,L)\leq 2\bomega(G)$.
\end{proof}

Combining \Cref{lem:st-to-wk} with \Cref{thm:upperbound} we get the following.
\begin{corollary}
\label{cor:wcol}
	For every graph $G$ and every positive integer $r$ we have
\begin{equation*}
	\wcol_r(G)\leq \frac12\bigl((2\tww(G)+6)\bomega(G)\bigr)^r.
\end{equation*}
\end{corollary}
Note that the base of the exponential  comes from the degeneracy of $G$. 
In order to improve this upper bound it is thus natural to try to improve the degeneracy bound. Hence the following problem:
\begin{problem}
What is the maximum degeneracy of a $K_{s+1,s+1}$-free graph with twin-width at most $d$?
\end{problem}

Recall that a \emph{depth $r$ minor} of a graph $G$ is a graph $H$ obtained from $G$ by taking a subgraph and contracting vertex disjoint subgraphs of radius at most~$r$. 
The \emph{greatest reduced average density} (grad) of $G$ with rank $r$ is the maximum ratio $|E(H)|/|V(H)|$ over all (non-empty) depth $r$ minors of a graph $G$; it is denoted by $\nabla_r(G)$. Hence, by definition, a class $\mathscr C$ has bounded expansion if, for each positive integer $r$, we have $\sup\{\nabla_r(G): G\in\mathscr C\}<\infty$. 
It is known that $\nabla_r(G)\leq \wcol_{2r+1}(G)$~\cite{Zhu2008}.
Hence the next corollary directly follows from \Cref{cor:wcol}.
\begin{corollary}
For every graph $G$ and every positive integer $r$ we have
\begin{equation}
	\nabla_r(G)\leq \frac12((2\tww(G)+6)\bomega(G))^{2r+1}.	
\end{equation}
\end{corollary}
In particular, every class of graphs of bounded clique-width that exclude a biclique as a subgraph has (at most) exponential expansion.

\section{Lower bounds}

It is known that high-girth graphs have large strong coloring numbers~\cite{Grohe2018}.
On the other hand, there exist expander graphs with high girth and small twinwidth~\cite{twin-width2}.
We combine both results to construct graphs with small twinwidth whose strong $r$-coloring numbers grow exponentially in $r$.

\begin{proposition}[{\cite[Theorem 5.1]{Grohe2018}}]\label{prop:scol-girth}
Let $G$ be a $d$-regular graph of girth at least $4g+1$, where $d\ge 7$. Then for every $r \le g$,
\[
    \col_r(G)\geq \frac{d}{2}\biggl(\frac{d-2}{4}\biggr)^{2^{\lfloor \log_2 r \rfloor}-1}.
\]
\end{proposition}

\begin{lemma}
\label{lem:bigcol}
	For every integer $\Delta \geq 7$ and every integers $r$ and $g\geq 4r+1$ there exists a $\Delta$-regular graph $G$ with girth at least $g$, $2\Delta-1\leq\tww(G) \le 2\Delta$, and
\[
\col_r(G))\geq \frac{\Delta}{2}\biggl(\frac{\Delta-2}{4}\biggr)^{2^{\lfloor \log_2 r\rfloor}-1}\geq \frac{\tww(G)}{4}\biggl(\frac{\tww(G)-4}{8}\biggr)^{2^{\lfloor \log_2 r\rfloor}-1}.
\]
\end{lemma}
\begin{proof}
We will construct a sequence $G_0,G_1,\dots$ of $\Delta$-regular graphs of twin-width at most $2\Delta$ and increasing girth.
Once we reach a graph with girth at least $g\geq 4r+1$, the result of this lemma follows from \Cref{prop:scol-girth}. Note that the twin-width of a $\Delta$-regular graph with girth at least $5$ is at least $2\Delta-1$ (because of the first contraction).

We define $G_0$ to be the complete graph with $\Delta+1$ vertices.
We fix a graph $G_{k-1}$ with edges $e_1,\dots,e_m$ and describe how to construct $G_k$.
For every edge $e_i$ of $G_{k-1}$, we define a mapping $\theta_{e_i} \colon \{0,1\}^m \to \{0,1\}^m$
that flips the $i$th coordinate and preserves all other coordinates, i.e.,
$\theta_{e_i}((x_1,\dots,x_m)) = (y_1,\dots,y_m)$ with
\[
    y_j=\begin{cases}
    1-x_j&\text{if }j=i\\
    x_j&\text{otherwise}.
    \end{cases}
\]
We define $G_{k}$ to be the graph with $V(G_k)=V(G_{k-1})\times \{0,1\}^m$ and $E(G_k)=\{\{(u,\mathbf{x}),(v,\theta_{uv}(\mathbf{x})\} : uv\in E(G_{k-1})\text{ and }\mathbf x \in \{0,1\}^m\}$.

A \emph{$2$-lift} of a graph $G$ is a graph obtained by adding for every vertex $v$ of $G$ two vertices $v_1$ and $v_2$, and adding for every edge $uv$ of $G$
either the edges $u_1v_1$ and $u_2v_2$ (parallel edges) or the edges $u_1v_2$ and $u_2v_1$ (crossing edges).
The graph $G_{k+1}$ can be obtained by a sequence of $2$-lifts from $G_k$
and therefore also by a sequence of $2$-lifts from $G_0 = K_{\Delta+1}$.
We can construct a contraction sequence that ``undoes`` these $2$-lifts by repeatedly contracting all pairs of duplicates.
Once we reach $K_{\Delta+1}$, we simply contract the remaining vertices one by one.
While doing so, the red degree never exceeds $2\Delta$ (see also \cite[Lemma 26]{twin-width2}).
Hence $\tww(G_k)\leq 2\Delta$.

It remains to show that the girth of $G_{k}$ is higher than the girth of $G_{k-1}$.
Let $\gamma$ be a shortest cycle of $G_k$. Let $p_V:V(G_{k})\rightarrow
V(G_{k-1})$ be the standard projection, and let $p_E:E(G_{k})\rightarrow E(G_{k-1})$
be the associated projection. It is easily checked that applying $p_E$
to a cyclic graph yields a cyclic graph, and thus $p_E(\gamma)$ includes a cycle.
If we apply the composition of all the mappings $\theta_{p_E(e)}$ for $e\in\gamma$ then the
starting vertex of $\gamma$ is fixed. It follows that each $\theta_e$ is
applied an even number of times. Thus the length of $\gamma$ is at least twice
the length of $p_E(\gamma)$.
Hence, the girth of $G_k$ is at least twice the girth of $G_{k-1}$. 
\end{proof}

\begin{corollary}\label{cor:lowerboundscol}
	For every integer $d\geq 14$, every positive integer $s$, and every integer $r$ of the form $2^k$, there exists a graph $G$ with $\tww(G)\leq d$, $\bomega(G)=s$, and 
\[
    \col_r(G)\geq \frac{ds}{4}\biggl(\frac{d-4}{8}\biggr)^{r-1}\geq 2\biggl(\frac{\tww(G)-4}{8}\biggr)^r\bomega(G).
\]
\end{corollary}
\begin{proof}
Take the lexicographic product of a graph obtained by \Cref{lem:bigcol} and $K_{s}$. This way we get a graph with twin-width at most $d\geq 14$ and no $K_{s+1,s+1}$.
\end{proof}

\begin{remark}
	The $2$-lift construction used in the proof of \Cref{lem:bigcol} was used in \cite{twin-width2} to prove that there exist cubic expander graphs with twin-width at most $6$. It follows from this result and the characterization of classes with polynomial expansion~\cite{dvorak2016strongly}
	that for $d \geq 6$, 
	the value $\nabla_r(G)$ is not bounded on 	the $K_{2,2}$-free graphs with twin-width at most $d$ by a polynomial function of $r$.
	We leave as a question wether 
	$\sup\{\nabla_r(G): \tww(G)\leq d\text{ and }\bomega(G)\leq s\}$ increases exponentially with $r$ for sufficiently large~$d$.
\end{remark}

Admissibility being a lower bound for strong coloring numbers, the above results do not provide any lower bound for admissibility.
We show below how to construct classes of graphs that have no $K_{2,2}$-subgraph with low twin-width and high admissibility.
\begin{lemma}[{\cite[Proposition~28]{twin-width2}}]\label{lem:subdcliq0}
    For $d \geq 0$ and $k>0$, if the clique $K_n$ subdivided $k$ times has twin-width less than $d$, then $k \geq \log_{d}(n-1)-1$.
\end{lemma}

\begin{lemma}[{\cite[Proposition~31]{twin-width2}}]\label{lem:subdcliq}
    For any $c>0$, the class of cliques $K_n$ subdivided at least $\frac{\log n}{c}$ times has twin-width at most $f(c)$ for some triple exponential function~$f$.
\end{lemma}

\begin{lemma}\label{lem:lowbo2}
For every integers $d$ and $r\geq 4$ there is a graph $G$ such that
\begin{itemize}
    \item $G$ has no $K_{2,2}$ subgraph;
    \item $\tww(G) \leq f(2 \log d)$;
    \item $\adm_r(G) \geq d^{2(r-1)}$,
\end{itemize}
where $f$ is the function of \Cref{lem:subdcliq}.
\end{lemma}
In particular, the above lemma implies that for every $d$, there is a graph class of bounded twin-width, whose members contain no $K_{2,2}$, but which has $r$-admissibility (and thus $r$-weak and strong coloring numbers) at least $d^{2(r-1)}$.
\begin{proof}
Let $d,r \in \N$ and $n = d^{2(r-1)}$. We define $G^d_r$ as the graph obtained by subdividing $r-1$ times each edge of $K_n$.
By construction, $G^d_r$ has no $K_{2,2}$ subgraph.
Let $c = 2 \log d$ and notice that $r-1 = \frac{\log n}{c}$.
According to \Cref{lem:subdcliq}, $G^d_r$ has twin-width at most $f(c)$.
On the other hand, as $r > 3$ we have
\begin{align*}
    r-1 & < 2(r-1) - 2\\
        & < \log_{d}\left ( d^{2(r-1)}\right ) - 2\\
        & < \log_{d}\left ( n - 1\right ) - 1.
\end{align*}

In order to prove the bound on admissibility, let us now consider an arbitrary ordering $\sigma$ of $V(G^d_r)$. Notice that $G^d_r$ has two types of vertices: $n$ vertices of degree $n-1$, which correspond to the vertices of the $n$-clique that was used to construct $G^d_r$, and vertices of degree 2, which have been introduced by subdivisions. Let $x$ denote the vertex of degree $n-1$ that appears the latest in $\sigma$. Notice that there are $n-1$ paths of length $r$ that start in $x$ and are otherwise disjoint and end at the $n-1$ other vertices of degree $n-1$ of $G$. By definition of $x$, all these vertices appear before in the ordering. This implies $\adm_r(G^d_r, \sigma) \geq n = d^{2(r-1)}$. As $\sigma$ was chosen arbitrarily, the same bound holds for the $r$-admissibility of $G^d_r$.
\end{proof}

\begin{corollary}
    For all integers $d$, $r\geq 4$, there is a constant $\varepsilon>0$ such that for all positive integers $r$ and $n$ there is a    
    $K_{s+1,s+1}$-free graph $G$ with 	$|G|\geq n$, $\bomega(G)=s$,  
   and  
  \[\adm_r(G)\geq (\log\log \tww(G))^\varepsilon\,\bomega(G).\]
    
\end{corollary}
\begin{proof}
	We first consider the case where $s=1$.
     Let $G_0$ be the graph given by \Cref{lem:lowbo2}.
     If $|G_0| \geq n$ then $G_0$ is the desired graph.
     Otherwise, we denote by $G$ the disjoint union of $n$ copies of $G_0$. Clearly this does not create any $K_{2,2}$ thus $\bomega(G_0)=1$. We have $\adm_r(G)\geq d^{2(r-1)}$, as otherwise any ordering of $G$ with smaller $r$-admissibility would give an ordering with smaller $r$-admissibility for $G_0$. Finally, as the twin-width of the disjoint union of two graphs is the maximum of the twin-width of each of them, we have  $\tww(G) = \tww(G_0)$, so $d \leq \tww(G)\leq f(2 \log d)$.
     The existence of the constant $C$ then follows from the fact that $f$ is a triple exponential function.
     
    The case where $s>1$ then follows by considering the lexicographic product of the graphs obtained above by $K_{s}$.
\end{proof}

\section*{Acknowledgments}
These results have been obtained at the occasion of the workshop ``Generalized coloring numbers and friends'' of the Sparse Graphs Coalition, which was co-organized by Micha{\l} Pilipczuk and Piotr Micek. 

\providecommand{\bysame}{\leavevmode\hbox to3em{\hrulefill}\thinspace}
\providecommand{\MR}{\relax\ifhmode\unskip\space\fi MR }
\providecommand{\MRhref}[2]{%
  \href{http://www.ams.org/mathscinet-getitem?mr=#1}{#2}
}
\providecommand{\href}[2]{#2}

\end{document}